\documentclass[11pt]{amsart}

\usepackage{hyperref}
\usepackage{amsfonts,mathrsfs,bbm,latexsym,rawfonts,amsmath,amssymb,amsthm,epsfig}
\usepackage{fullpage, setspace, color}
\newtheorem{thm}{Theorem}[section]
\newtheorem{cor}[thm]{Corollary}
\newtheorem{rem}[thm]{Remark}
\newtheorem{lem}[thm]{Lemma}

\numberwithin{equation}{section}

 \newcommand{\al}{\alpha}

 \newcommand{\Ld}{\Lambda}

 \newcommand{\ep}{\varepsilon}
 \newcommand{\Si}{\Sigma}

 \newcommand{\ga}{\gamma}
 \newcommand{\Ga}{\Gamma}

 \newcommand{\sch}{Schr\"odinger }

 \newcommand{\G}{\mathcal{G}}

 \DeclareMathOperator{\tr}{tr}

 \newcommand{\Real}{\mathbb{R}}
 
 \newcommand{\Complex}{\mathbb{C}}

 \def\<{\left\langle} \def\>{\right\rangle}
 \def\({\left(} \def\){\right)}
 \newcommand{\n}{\nabla}
 \newcommand{\p}{\partial}

 \renewcommand{\t}[1]{\tilde{#1}}

 \renewcommand{\vec}[1]{\textbf{#1}}
 \renewcommand{\H}{\vec{H}}
 \newcommand{\A}{\vec{A}}

 \newcommand{\nt}{\n^\top}
 \newcommand{\nn}{\n^\bot}
  \newcommand{\tnt}{\tn^\top}
 \newcommand{\tnn}{\tn^\bot}
  \newcommand{\pit}{\pi^\top}
 \newcommand{\pin}{\pi^\bot}

  \newcommand{\Gt}{\G^\top}
 \newcommand{\Gn}{\G^\bot}
   \newcommand{\gt}{g^\top}
 \newcommand{\gn}{g^\bot}
 
 \newcommand{\tn}{\tilde{\nabla}}
 
  \newcommand{\bg}{\bar{g}}
 \newcommand{\tg}{\tilde{g}}
 \newcommand{\tgt}{\tg^\top}
 \newcommand{\tgn}{\tg^\bot}
 
 \newcommand{\tJ}{\t{J}}

\begin{document}

\title{Gauss map of the skew mean curvature flow}

\author{Chong Song}

\address{School of Mathematical Sciences, Xiamen University, Xiamen, 361005, P.R.China.}
\email{songchong@xmu.edu.cn}



\begin{abstract}
  The skew mean curvature flow(SMCF) is a natural generalization of the famous vortex filament equation. In this note, we show that the Gauss map of the SMCF satisfies a \sch flow equation. In this regard, we explore the geometry of the oriented Grassmannian manifold explicitly by embedding it into the exterior product space.
\end{abstract}

\maketitle

\section{Introduction}

The \emph{skew mean curvature flow}(SMCF) is an evolution equation defined on a codimension two submanifold, along which the submanifold deforms in its binormal direction. The SMCF origins from hydrodynamics and models the locally induced motion of \emph{vortex membranes}, i.e. singular vortices supported on codimension two subsets~\cite{Sh,Kh}. It also appears in the context of superfluidity and superconductivity describing the asymptotic behavior of vortices which is governed by a complex-valued PDE~\cite{J}. In particular, the SMCF is a generalization of the famous \emph{vortex filament equation}(VFE)~\cite{Da} in both higher dimensions and Riemannian geometry. For a detailed introduction, we refer to~\cite{SS} and references therein.

The SMCF can be regarded as the Hamiltonian flow of the volume functional on the infinite dimensional symplectic space of codimension two submanifolds~\cite{HV,SS}. There is another famous Hamiltonian flow, i.e. the \emph{\sch flow}, which corresponds to the energy functional defined on the space of maps from a Riemannian manifold into a symplectic manifold~\cite{DW1}. It is well-known that the evolution equation of the Gauss map of the VFE coincides with the one dimensional \sch flow on the standard sphere. In fact, suppose $\ga:[0,T]\times S^1\to \Real^3$ satisfies the VFE
\[\p_t\ga=\p_s\ga\times\p_s^2\ga=k\mathbf{b},\]
where $s$ is the arc-length parameter, $k$ is the curvature, $\mathbf{b}$ is the binormal vector and $\times$ denotes the cross product in $\Real^3$. Then it is straightforward to verify that the Gauss map $u:=\p_s\ga:[0,T]\times S^1\to S^2$ satisfies the \sch flow
\[\p_t u=u\times \Delta u=J(u)\tau(u),\]
where $J$ is the complex structure on $S^2$ and $\tau(u)$ is the tension field.

In this note, we show that this relationship holds for higher dimensions, namely, the Gauss map of a solution to the SMCF satisfies a \sch flow equation.

More precisely, suppose $\Si$ is an $m$ dimensional oriented manifold and $F:[0,T]\times \Si\to \Real^{m+2}$ is a family of immersions in the $(m+2)$ dimensional Euclidean space. For each $t\in [0,T]$, the normal bundle of the submanifold $F(t, \Si)$ admits an induced complex structure $J(F)$ which simply rotates a vector in the normal plane by $\pi/2$ in the positive direction. Then the SMCF is defined by
\begin{equation}\label{e:SMCF}
         \p_t F=J(F)\textbf{H}(F),
\end{equation}
where $\textbf{H}(F)$ is the mean curvature of the submanifold.

Recall that the Gauss map of a submanifold sends a point to its tangent plane, which lies in the Grassmannian manifold. Thus the Gauss map of the family of immersions $F$ gives rise to a time-dependent map $\rho$ from $\Sigma$ to the (oriented) Grassmannian $G(m,2)$. It is well-known that $G(m,2)$ is a K\"ahler manifold which admits a canonical complex structure $\t{J}$ (see for example \cite{HO}). The tension field of $\rho$ is defined by
\[\tau(\rho):=\tr_{g_t}\tn d\rho,\]
where $\tn$ is the pull-back of the Levi-Civita connection on $G(m,2)$ and the trace is taken with respect to the induced metric $g_t$ on $\Si$. Our main result is the following

\begin{thm}\label{t:Gauss}
The Gauss map $\rho:[0,T]\times \Si\to G(m,2)$ of a solution to the SMCF (\ref{e:SMCF}) satisfies the Schr\"odinger flow
\begin{equation}\label{e:gauss}
\p_t\rho=\t{J}(\rho)\tau(\rho).
\end{equation}
\end{thm}
\begin{rem}
The difference between the flow (\ref{e:gauss}) and the usual \sch flow is that the metric $g_t$ on the domain manifold $\Sigma$ are evolving along the SMCF, while for the usual \sch flow, the domain is a fixed Riemannian manifold.
\end{rem}

A classical result of Ruh and Vilms~\cite{RV} shows that the Gauss map of a minimal submanifold is harmonic. Wang~\cite{Wang} proved that the Gauss map of the \emph{mean curvature flow}(MCF) satisfies the harmonic map heat flow. Since the SMCF is the Hamiltonian counterpart of the MCF and the \sch flow is the Hamiltionian counterpart of the harmonic map heat flow, Theorem~\ref{t:Gauss} is parallel and consistent with previous results. It is interesting that Terng and Uhlenbeck~\cite{TU} investigated the \sch map into complex Grassmannians, while Theorem~\ref{t:Gauss} proposes a rather natural \sch flow into the Grassmannian $G(m,2)$.

To provide a self-contained proof to Theorem\ref{t:Gauss}, we first explore the geometry of Grassmannian manifolds in general. By embedding the Grassmannian $G(m,k)$ into the exterior product space $\Ld^m\Real^{m+k}$, we study the product structure of its tangent space and express its Levi-Civita connection explicitly. Moreover, for the codimension two Grassmannian $G(m,2)$, we show the equivalence of its canonical complex structure $\t{J}$ and the complex structure $J$ on the normal bundle of the underlying submanifold. The proof of Theorem~\ref{t:Gauss} is given in Section~\ref{s:proof}. We remark that our argument here also gives an alternative proof of the parallel statements on minimal submanifolds and MCF.

\subsection*{Acknowledgement}

Part of this work was completed while the author was visiting University of British Columbia and University of Washington. The author would like to thank Professor Jingyi Chen and Yu Yuan for their support and encouragement. He would also like to thank Professor Youde Wang and Jun Sun for inspiring discussions. The work is partially supported by the Natural Science Foundation of Fujian Province of China (No. 2014J01023) and the Fundamental Research Funds for the Central Universities.

\section{Grassmannian manifolds and Gauss map}

\subsection{Natural embedding of Grassmannian manifold}

Let $m,k,n=m+k$ be positive integers. Let $\Real^{n}$ be the $n$-dimensional Euclidean space with a fixed orientation. The Grassmannian $G:=G(m,k)$ is by definition the set of $m$-dimensional \emph{oriented} subspaces in $\Real^n$. Given a plane $\xi\in G$, we denote the orthogonal complement $k$-plane of $\xi$ by $\xi^\bot$, which has a natural orientation induced from that of $\xi$.

Let $\Ld^m\Real^n$ be the linear space of $m$-vectors, which is endowed with the standard metric $\bg$ induced from $\Real^n$. The Grassmannian $G$ can be realized as the set of unit simple $m$-vectors, which is an embedded submanifold of $\Ld^m\Real^n$ . In fact, for each $m$-plane $\xi \in G$, we choose arbitrarily an oriented orthonormal basis $\{e_1,\cdots,e_m\}\subset \xi$ and let
\[\sigma(\xi)=e_1\wedge\cdots\wedge e_m\in \Ld^m\Real^n.\]
Since the wedge product of a different orthonormal basis (with the same orientation) is the same, the embedding $\sigma:G\to \Ld^m\Real^n$ is well-defined. Note that actually $\sigma(G)$  lies in the unit sphere in $\Ld^m\Real^n$ and hence is compact.

We will call the induced metric $\tilde{g}:=\sigma^*\bg$ the \emph{canonical metric} on the Grassmannian manifold $G$. The corresponding Levi-Civita connection will be denoted by $\tn$. Given an $m$-plane $\xi\in G$, an oriented orthonormal basis $\{e_1,\cdots, e_m, \nu_1, \cdots, \nu_k\}$ of $\Real^n$ is called \emph{adapted} if $\xi=span\{e_1,\cdots, e_m\}$ and $\xi^\bot=span\{\nu_1, \cdots, \nu_k\}$. In the following, we will not distinguish an $m$-plane $\xi$ from the corresponding $m$-vector $e_1\wedge\cdots\wedge e_m$.

\subsection{The tangent space of $G$}

Suppose $\eta:(-\ep, \ep)\to G$ is a smooth curve with $\eta(0)=\xi \in G$. We can find a time-dependent orthonormal frame $\{e_1(t),\cdots, e_m(t), \nu_1(t), \cdots, \nu_k(t)\}$ which is adapted to the plane $\eta(t)$ at each time $t\in (-\ep, \ep)$. Using the embedding of $G$ in $\Ld^m\Real^n$, we can express the curve $\eta$ by
\begin{equation}\label{e:101}
  \eta(t)=e_1(t)\wedge\cdots\wedge e_m(t).
\end{equation}
Let $\pit$ and $\pin$ be the projections from $\Real^n$ into $\xi$ and $\xi^\bot$, respectively. By differentiating (\ref{e:101}), we get
\[\begin{aligned}
  \frac{d}{dt}\eta\Big|_{t=0}&=\left(\frac{d}{dt}e_1\wedge\cdots\wedge e_m+\cdots +e_1\wedge\cdots\wedge \frac{d}{dt}e_m\right)\Big|_{t=0}\\
  &=\tnn_t e_1\wedge\cdots\wedge e_m\Big|_{t=0}+\cdots +e_1\wedge\cdots\wedge \tnn_t e_m\Big|_{t=0},
\end{aligned}\]
where $\tnn_t:=\pin\circ \frac{d}{dt}$. It follows that the tangent place of $T_{\xi} G$ can be spanned by the orthonormal basis
\begin{equation}\label{e:basis}
  \{E_{i\al}:=e_1\wedge\cdots\wedge\underbrace{\nu_\al}_i\wedge\cdots\wedge e_m;1\le i\le m, 1\le \al\le k\}.
\end{equation}

\subsection{The tautological bundle}

There is a so-called \emph{tautological bundle} $\Gt$ over $G$ (also referred to as the universal or canonical bundle, see for example~\cite{KN}) obtained by simply attaching the $m$-plane $\xi$ itself at each point $\xi\in G$.
The inclusion $\iota_\xi:\xi\to \Real^n$ at each plane $\xi\in \Gt$ naturally induces a metric $\tgt$ on the bundle $\Gt$. We can also define a connection $\tnt:\Ga(TG)\times \Ga(\Gt)\to \Ga(\Gt)$ on $\Gt$ as follows. Let $\xi\in G$ be an $m$-plane and $\pit:\Real^n\to\xi$ be the orthogonal projection. For any tangent vector $X\in T_\xi G$, we can find a curve $\eta:(-\ep,\ep)\to G$ such that $\eta'(0)=X$. Now given a section $u\in \Ga(\Gt)$, we define
\[ \tnt_Xu|_\xi=\pit\circ \frac{d}{dt}\Big|_{t=0}\left(\iota_{\eta(t)}\circ u(\eta(t))\right).\]
Its easy to verify that the connection $\tnt$ is compatible with the metric $\tgt$, i.e. $\tnt\tgt=0$.

There is another bundle $\Gn$ over $G$ obtained by attaching the orthogonal complement $k$-plane $\xi^\bot$ at each $\xi\in G$. Similarly, we have a natural metric $\tgn$ and a compatible connection $\tnn$ on the bundle $\Gn$.

The above two bundles together give us a tensor product bundle $\G:=\Gt\otimes \Gn$ over $G$ whose fiber at $\xi$ is the tensor product $\xi\otimes \xi^\bot$. We can construct a tensor product metric $\tgt\otimes \tgn$ and a tensor product connection $\tnt\otimes \tnn$ on $\G$. Namely, for any $u\otimes \mu, v\otimes \nu\in \Ga(\G)$, we have
\[\begin{aligned}
(\tgt\otimes \tgn)(u\otimes \mu, v\otimes \nu)&=\tgt(u,v)\cdot\tgn(\mu,\nu),\\
(\tnt\otimes \tnn)(u\otimes \mu)&=\tnt u\otimes \mu+u\otimes \tnn \mu.
\end{aligned}\]
Obviously, the connection is compatible with the metric since
\[ (\tnt\otimes \tnn)(\tgt\otimes \tgn)=\tnt\tgt\otimes \tgn+\tgt\otimes \tnn\tgn=0.\]

\subsection{The tensor product structure of $TG$}

Given an adapted frame $\{e_1,\cdots, e_m, \nu_1, \cdots, \nu_k\}$ at $\xi\in G$, we have an orthonormal basis $\{E_{i\al}, 1\le i\le m, 1\le \al\le k\}$ of $T_\xi G$ which is defined by (\ref{e:basis}). Using the duality w.r.t. the $m$-form $e_1^*\wedge\cdots\wedge e_m^*$, we may identify $e_1\wedge\cdots\wedge\widehat{e_i}\wedge\cdots\wedge e_m$ with $e_i^*$ and hence identify $E_{i\al}$ with  $e_i^*\otimes \nu_\al$. This gives a linear map $\phi_\xi: \xi^*\otimes \xi^\bot\to T_\xi G$ at each $\xi\in G$ such that $\phi_\xi(e_i^*\otimes \nu_\al)=E_{i\al}$. It is obvious that $\phi_\xi$ is independent of the choice of the basis and is an isometry. By identifying $\xi$ with its dual space $\xi^*$, we also have an isometry $\psi_\xi: \xi\otimes \xi^\bot\to T_\xi G$ such that $\psi_\xi(e_i\otimes \nu_\al)=E_{i\al}$. Consequently, we can define a bundle map $\Psi:\G\to TG$ where the fiber-wise map at each point $\xi\in G$ is given by $\psi_\xi$.

\begin{thm}\label{t:bundle-equivalence}
  The bundle map $\Psi:(\G, \tgt\otimes \tgn, \tnt\otimes \tnn) \to (TG, \tg, \tn)$ is a bundle isometry which preserves the connection, i.e.
  \begin{equation}\label{e:bundle-equivalence}
    \Psi^*\tg=\tgt\otimes \tgn, \quad \Psi^*\tn=\tnt\otimes \tnn.
  \end{equation}
\end{thm}
\begin{proof}
The statement that the map $\Psi$ is an isometry is obvious since each fiber-wise map $\psi_\xi$ is an isometry. So we only need to prove the second equality in (\ref{e:bundle-equivalence}), and it suffices to prove it at one point. Thus fixing a point $\xi\in G$, we only need to show that for any tangent vector $X\in T_{\xi}G$ and a local adapted frame defined near $\xi$,
\begin{equation}\label{e:102}
  \Psi((\tnt\otimes \tnn)_X (e_i\otimes \nu_\al))=\tn_X (\Psi(e_i\otimes \nu_\al))=\tn_X E_{i\al}.
\end{equation}

To this order, we choose a curve $\eta:(-\ep,\ep)\to G$ such that $\eta(0)=\xi$ and $\eta'(0)=X$. Let $\tilde{\pi}:\Ld^m\Real^n\to T_{\xi} G$ be the projection to the tangent space of $\xi$. Since $G$ is a submanifold of $\Ld^m \Real^n$ and $\tn$ is the induced Levi-Civita connection, by definition we have
\[\begin{aligned}
\tn_X E_{i\al} &= \tilde{\pi}\circ \frac{d}{dt}\Big|_{t=0}(e_1\wedge\cdots\wedge\underbrace{\nu_\al}_{i}\wedge\cdots\wedge e_n)\\
&=e_1\wedge\cdots\wedge\underbrace{\tnn_t\nu_\al}_{i}\wedge\cdots\wedge e_n + \sum_{j\neq i}e_1\wedge\cdots\wedge\underbrace{\nu_\al}_{i}\wedge\cdots\wedge\underbrace{\tnt_t e_j}_j\wedge\cdots\wedge e_n\\
&=e_1\wedge\cdots\wedge\underbrace{\tnn_t\nu_\al}_{i}\wedge\cdots\wedge e_n - \sum_{j\neq i}(\tnt_t e_j, e_i)e_1\wedge\cdots\wedge\underbrace{e_i}_{i}\wedge\cdots\wedge\underbrace{\nu_\al}_{j}\wedge\cdots\wedge e_n.
\end{aligned}\]
In the last identity, we used the fact that $\tnt_t e_j=\sum_{k\neq j}(\tnt_t e_j, e_k)e_k$ and switched its position with $\nu_\al$.
Then under the bundle map, we have
\[\begin{aligned}
    \Psi^{-1}(\tn_X E_{i\al}) &=e_i\otimes \tnn_t\nu_\al - \sum_{j\neq i}(\tnt_t e_j, e_i)e_j\otimes \nu_\al\\
  &=e_i\otimes \tnn_t\nu_\al + \sum_{j\neq i}(e_j, \tnt_t e_i)e_j\otimes \nu_\al\\
  &=e_i\otimes \tnn_t\nu_\al+\tnt_te_i\otimes \nu_\al\\
  &=(\tnt\otimes \tnn)_X(e_i\otimes \nu_\al).
\end{aligned}\]
Therefore (\ref{e:102}) holds true and the proof is finished.
\end{proof}

\subsection{Submanifold and Gauss map}

Suppose $\Si$ is an oriented $m$-dimensional manifold. Let $f:\Sigma\to \Real^n$ be an immersion and $M:=f(\Sigma)$ be the immersed submanifold. The Gauss map $\rho$ is a map from $\Si$ to the Grassmannian $G$ which sends each point $x\in \Si$ to the tangent plane $T_{f(x)}M$.

The restriction of the ambient tangent bundle $T\Real^n$ on $M$ splits into the tangent bundle $TM$ and the normal bundle $NM$. We denote the naturally induced metric on $TM$ by $\gt$ and corresponding connection by $\nt$, and denote the induced metric on $NM$ by $\gn$ and corresponding connection by $\nn$.

Using the immersion map $f$, we have pull-back bundles $f^*TM$ and $f^*NM$ over $\Si$, which are endowed with corresponding pull-back metrics and connections. On the other hand, using the Gauss map $\rho$, we can define pull-back bundles $\rho^*\Gt$ and $\rho^*\Gn$ with pull-back metrics and connections. It is straightforward to see that these bundles are identical, i.e.
\[f^*TM\cong\rho^*\Gt, \quad f^*NM\cong \rho^*\Gn. \]
It follows from Theorem~\ref{t:bundle-equivalence} that there is a bundle isometry between $f^*TM\otimes f^*NM$ and $\rho^*TG$ which preserves the connection. In summary, we have

\begin{cor}\label{c:pull-back-equivalence}
Suppose $f:\Si\to \Real^n$ is an immersion and $\rho:\Si \to G$ is the corresponding Gauss map, then we have following bundle identifications
  \begin{align*}\label{e:pull-back-equivalence}
  &f^*(TM, \gt, \nt)\cong \rho^*(\Gt, \tgt,\tnt),\\
  &f^*(NM, \gn, \nn)\cong \rho^*(\Gn, \tgn,\tnn),\\
  &f^*(TM\otimes NM, \gt\otimes \gn, \nt\otimes\nn)\cong \rho^*(TG, \tg,\tn).
  \end{align*}
\end{cor}

In particular, we have the following

\begin{lem}\label{l:identify}
Under the above identifications, the differential $d\rho$, which is a section of $\rho^*TG\otimes T\Si$, is equivalent to the second fundamental form $\A$ of the immersion, which is a section of $f^*NM\otimes T\Si\otimes T\Si$.
\end{lem}
\begin{proof}
Choose an open neighborhood in $\Si$. Let $\{e_1, \cdots, e_m, \nu_1, \nu_2\}$ be a local adapted frame on the pull-back bundle $f^*(T\Real^n|_M)$. Let $\ep_i$ be the corresponding tangent vector on $\Si$ such that $f_*\ep_i=e_i$. By definition, the Gauss map is given by $\rho(x)=e_1(x)\wedge \cdots\wedge e_m(x)$ and its differential is
\[\begin{aligned}
  d\rho&=\p_{i}\rho\otimes \ep_i^*\\
  &=(e_1\wedge\cdots\wedge\nn_ie_j\wedge\cdots\wedge e_m)\otimes \ep_i^*\\
  &=(\nn_ie_j, \nu_\al)E_{j\al}\otimes \ep_i^*.
\end{aligned}\]
Using the bundle isomorphism in Corollary~\ref{c:pull-back-equivalence}, we can identify $E_{j\al}$ with $e_j\otimes \nu_\al$ and hence the above term as
\[ (\n_ie_j, \nu_\al)\nu_\al\otimes\ep_j^*\otimes \ep_i^*, \]
which is exactly the second fundamental form $\A$.
\end{proof}

\section{Gauss Map of SMCF}

\subsection{Complex structure on normal bundle}

In the following, we set $k=2$ and $n=m+2$. Suppose $\Si$ is an oriented $m$-dimensional manifold. Let $\Real^n$ be the Euclidean space which is endowed with an orientation and let $f:\Si\to\Real^n$ be a codimension two immersion. Denote the inner product of $\Real^n$ by $\<\cdot,\cdot\>$ and its standard connection by $D$.  We can define a natural complex structure $J$ on the normal bundle of the submanifold $M:=f(\Si)$ by simply rotating a vector in the normal plane by $\pi/2$ in the positive direction. More precisely, for any point $y\in M$ and normal vector $\nu\in N_yM$, we require $\<J\nu, \nu\>=0$, $|J\nu|=|\nu|$ and $f(\ep_1)\wedge\cdots\wedge f(\ep_m)\wedge \nu\wedge J\nu$ to be consistent with the chosen orientation of the ambient space $\Real^n$, where $\{\ep_1,\cdots, \ep_m\}$ is an oriented basis of $\Si$. The following simple fact about the complex structure actually holds for any Riemannian ambient space, see~\cite{SS}.

\begin{lem}\label{l:parallel}
The complex structure is parallel w.r.t. the normal connection, i.e. $\nn J = 0$.
\end{lem}
\begin{proof}
It suffices to show that for any locally supported unit normal vector field $\mu\in\Ga(NM)$ in the normal bundle, we have
\begin{equation*}
    J\nn \mu=\nn J\mu.
\end{equation*}
Set $\nu=J\mu$, then $\mu=-J\nu$ and $\{\mu, \nu\}$ forms a local orthonormal frame of the normal bundle. Then for any tangent vector field $u$, we have
\begin{equation*}
\begin{aligned}
J \nn_u \mu   &=  J(D_u\mu)^{\bot}\\
    &=J\left(\<D_u\mu,\mu\>\mu
    +\<D_u\mu,\nu\>\nu\right)\\
   &= -\<D_u\mu,\nu\>\mu,
\end{aligned}
\end{equation*}
while
\begin{equation*}
\begin{aligned}
 \nn_u(J\mu)&=\nn_u \nu=(D_u\nu)^{\bot}\\
    &=\<D_u\nu,\mu\>\mu
    +\<D_u\nu,\nu\>\nu\\
    &= -\<D_u\mu,\nu\>\mu.
\end{aligned}
\end{equation*}
This proves the lemma.
\end{proof}

\subsection{K\"ahler structure on $G(m,2)$}

It is well-known that the codimension two Grassmannian $G:=G(m,2)$, which is dual to $G(2,m)$, is a K\"ahler manifold. Actually, it can be modelled as a complex variety determined by a quadric in $\Complex P^{n-1}$(cf. \cite{HO}). Here we show that the K\"ahler structure of $G$ is naturally induced by the complex structure on the normal bundle.

Recall that we have a natural vector bundle $\Gn$ over $G$ whose fiber at each point $\xi\in G$ is the orthogonal complement sub-space $\xi^\bot$. There is a natural complex structure $J$ on $\Gn$ which simply rotates a vector in $\xi^\bot$ by $\pi/2$ positively (w.r.t. the orientation of $\xi^\bot$). By a similar argument as in the proof of Lemma~\ref{l:parallel}, one can verify that $J$ is compatible with the canonical connection $\tnn$ on $\Gn$, i.e. $\tnn J=0$. In fact, the complex structure on the normal bundle of a submanifold is just the pull-back of the complex structure on $\Gn$ by its Gauss map.

On the other hand, we have the identification $TG\cong \Gt\otimes \Gn$ by Theorem~\ref{t:bundle-equivalence}. Then with the help of an adapted frame, we can define a complex structure $\t{J}:=Id\otimes J$ on $TG$ by
\begin{equation}\label{e:J-equivalence}
 \t{J}E_{i\al}=\t{J}(e_i\otimes \nu_\al)=e_i\otimes (J\nu_\al).
\end{equation}
Obviously, $\t{J}$ is compatible with the Levi-Civita connection on $G$ since $\tn\t{J}=\tnn J=0$. Therefore, $(G,\tg,\t{J})$ is a K\"ahler manifold.

\subsection{Proof of Theorem~\ref{t:Gauss}}\label{s:proof}

Now we are ready to prove our main result on the Gauss map of the SMCF. Recall that $F:\Si\times[0,T)\to\Real^n$ is a family of codimension two immersions satisfying the SMCF (\ref{e:SMCF}). For simplicity, we denote the mean curvature by $\H:=\H(F)$ and the induced complex structure of the normal bundle by $J:=J(F)$. Then the SMCF is simply
\begin{equation}\label{e:SMCF1}
\p_t F=J\H.
\end{equation}
The Gauss map $\rho$ of $F$ is a time-dependent map from $\Si$ to the Grassmannian $G:=G(m,2)$ which is a K\"ahler manifold. Denote the complex structure of $G$ by $\tJ$ and its Levi-Civita connection by $\tn$. Our goal is to show that $\rho$ satisfies the \sch flow
\begin{equation}\label{e:05}
\p_t\rho = \tJ\tau(\rho).
\end{equation}

\begin{proof}[Proof of Theorem~\ref{t:Gauss}]
We only need to prove the identity (\ref{e:05}) at one point. Let $\{e_1, \cdots, e_m, \nu_1, \nu_2\}$ be a local adapted frame on the pull-back bundle $F^*(T\Real^n|_M)$. Let $\ep_i$ be the corresponding tangent vector on $\Si$ such that $F_*(\ep_i)=e_i$. Then
\[F_*[\p_t,\ep_i]=[\p_tF,e_i]=\p_t e_i-\p_i\p_tF.\]
Note that $F_*[\p_t,\ep_i]=F_*(\p_t\ep_i)$ is tangent to $M$. It follows
\[ (\p_t e_i-\p_i\p_tF)^\bot=0.\]
Hence
\[ \nn_i\p_t F = \nn_t e_i.\]
Using SMCF equation (\ref{e:SMCF}) and Lemma~\ref{l:parallel}, we have
\[ \nn_i\p_t F = \nn_i (J\H)= J\nn_i \H.\]
Thus we arrive at
\begin{equation}\label{e:04}
\nn_t e_i =J\nn_i \H.
\end{equation}
Now differentiating the Gauss map $\rho=e_1\wedge\cdots\wedge e_m$ and using (\ref{e:04}), we get
\[\begin{aligned}
    \p_t\rho &= (\nn_t e_1)\wedge\cdots\wedge e_m + \cdots + e_1\wedge\cdots\wedge (\nn_t e_m)\\
    &=(J\nn_i\H)\wedge\cdots\wedge e_m + \cdots + e_1\wedge\cdots\wedge (J\nn_i\H)\\
    &=(J\nn_i\H)\otimes e_i^*,
  \end{aligned}\]
where in the last identity we replace $e_1\wedge\cdots\wedge\widehat{e_i}\wedge\cdots\wedge e_m$ by $e_i^*$. Since the complex structure $J$ on the normal bundle $NM$ and the complex structure $\tJ$ on $G$ is equivalent by (\ref{e:J-equivalence}), we obtain
\begin{equation}\label{e:01}
\p_t\rho =\tJ(\nn_i\H\otimes e_i^*)=\tJ\nn \H.
\end{equation}

On the other hand, we can identify $dp$ with the second fundamental form $\A$ by Lemma~\ref{l:identify}. Moreover, the pull-back connection $\tn$ also coincides with the induced connection $\n$ by Corollary~\ref{c:pull-back-equivalence}. It follows that
\begin{equation}\label{e:02}
  \tau(\rho)=\tr \tn d\rho=\tr \n \A.
\end{equation}

Finally, recall that by the Codazzi equation, we have
\begin{equation}\label{e:03}
  \tr \n \A= \nn(\tr \A)=\nn \H.
\end{equation}
Combining (\ref{e:01}), (\ref{e:02}) and (\ref{e:03}), we conclude that $\rho$ satisfies (\ref{e:05}).

\end{proof}


\end{document}